\documentclass[a4paper, 12pt]{article}

\usepackage[utf8]{inputenc}
\usepackage[english]{babel}
\usepackage[T1]{fontenc}
\usepackage{adjustbox}  
\usepackage{graphicx}  
\usepackage{authblk}  
\usepackage[margin=1.2in]{geometry}  
\usepackage{xcolor, color}  
\usepackage[ruled,vlined]{algorithm2e}  
\usepackage{float}  
\usepackage{amsmath, amsthm, amssymb, mathrsfs}  
\usepackage{thmtools, thm-restate}  
\usepackage{enumerate}  
\usepackage{hyperref}  

\newcommand{\cc}[1]{\mathcal{#1}}  
\newcommand{\csf}[1]{\textsf{#1}}  

\definecolor{wisteria}{RGB}{142, 68, 173}  
\definecolor{asbestos}{RGB}{149, 165, 166}  
\definecolor{vanadyl}{RGB}{0, 151, 230}  
\definecolor{nephritis}{RGB}{39, 174, 96}  

\hypersetup{
	colorlinks=true,
	linkcolor=asbestos,
	filecolor=vanadyl,
	urlcolor=vanadyl,
	citecolor=wisteria
}

\newcommand{\userpar}[1]{\paragraph{#1}}  

\usepackage{tabularx}
\usepackage{environ}
\makeatletter
\newcommand{\problemtitle}[1]{\gdef\@problemtitle{#1}}
\newcommand{\probleminput}[1]{\gdef\@probleminput{#1}}
\newcommand{\problemquestion}[1]{\gdef\@problemquestion{#1}}
\NewEnviron{decproblem}{
	\problemtitle{}\probleminput{}\problemquestion{}
	\BODY
	\par\addvspace{.5\baselineskip}
	\noindent
	\begin{tabularx}{\textwidth}{@{\hspace{\parindent}} l X c}
		\multicolumn{2}{@{\hspace{\parindent}}l}{\@problemtitle} \\
		\textbf{Input:} & \@probleminput \\
		\textbf{Question:} & \@problemquestion
	\end{tabularx}
	\par\addvspace{.5\baselineskip}}
\NewEnviron{genproblem}{
	\problemtitle{}\probleminput{}\problemquestion{}
	\BODY
	\par\addvspace{.5\baselineskip}
	\noindent
	\begin{tabularx}{\textwidth}{@{\hspace{\parindent}} l X c}
		\multicolumn{2}{@{\hspace{\parindent}}l}{\@problemtitle} \\
		\textbf{Input:} & \@probleminput \\
		\textbf{Output:} & \@problemquestion
	\end{tabularx}
	\par\addvspace{.5\baselineskip}}

\theoremstyle{plain}  
\declaretheorem[name=Theorem]{theorem}
\declaretheorem[name=Lemma]{lemma}
\declaretheorem[name=Proposition]{proposition}
\declaretheorem[name=Corollary]{corollary}

\theoremstyle{definition}  

\declaretheorem[name=Observation]{observation}
\declaretheorem[name=Remark]{remark}

\declaretheorem[name=Open Problem]{problem}

\newcommand{\st}{:}  
\newcommand{\card}[1]{\vert #1 \vert}  
\newcommand{\pow}[1]{\mathbf{2}^{#1}}  
\newcommand{\cs}{\cc{C}}  
\newcommand{\cl}{h}  
\newcommand{\NP}{\csf{NP}}  
\newcommand{\comp}[1]{\overline{#1}}  
\newcommand{\eq}{\equiv_{AB}}  
\renewcommand{\H}{\cc{H}}  
\newcommand{\E}{\cc{E}}  
\DeclareMathOperator{\mfs}{MFS}  

\def\FIGLinkage{1.0}  
\def\FIGPreSatOne{1.0}  
\def\FIGPreSatTwo{1.0}  
\def\FIGSaturation{0.8}  
\def\FIGSatNotSuf{0.9}  
\def\FIGAlternative{0.75}  
\def\FIGClique{0.9}  
\def\FIGCobipNotSuf{0.8}  
\def\FIGEqClasses{0.8}  
\def\FIGEqGraph{0.95}  
\def\FIGEqGraphSep{0.9}  

\title{Half-space separation in monophonic convexity}
\author[1]{Mohammed Elaroussi}
\author[2]{Lhouari Nourine}
\author[3]{Simon Vilmin}

\affil[1]{Universit\'e de 
Bejaia, Facult\'e des Sciences Exactes,  Unit\'e de Recherche LaMOS, 06000 Bejaia, Algeria}
\affil[2]{Universit\'e Clermont-Auvergne, CNRS, Mines de Saint-\'Etienne, 
	Clermont-Auvergne-INP, LIMOS, 63000 Clermont-Ferrand, France.}
\affil[3]{Aix-Marseille Université, CNRS, LIS, Marseille, France.}

\begin{document}
\maketitle	

\begin{abstract} 
We study half-space separation in the convexity of chordless paths of a graph, i.e., monophonic convexity.
In this problem, one is given a graph and two (disjoint) subsets of vertices and asks whether these two sets can be separated by complementary convex sets, called half-spaces.
While it is known this problem is $\NP$-complete for geodesic convexity---the convexity of shortest paths---we show that it can be solved in polynomial time for monophonic convexity.

\vspace{0.5em}
\end{abstract}

\section{Introduction}
\label{sec:introduction}

A (finite) convexity space is a pair $(V, \cs)$ where $V$ is a (finite) groundset and $\cs$ a collection of subsets of $V$, called \emph{convex sets}, containing $\emptyset$, $V$ and closed under taking intersections.
Graphs provide a wide variety of different convexity notions, known as graph convexities.
These are usually defined based on paths and include for instance the geodesic convexity \cite{pelayo2013geodesic}, the monophonic convexity \cite{dourado2010complexity, duchet1988convex, farber1986convexity}, the $m^3$-convexity \cite{dragan1999convexity}, the triangle-path convexity \cite{changat1999triangle}, the toll convexity \cite{alcon2015toll}, or the weakly toll convexity \cite{dourado2022weakly}.

In this paper, we are interested in the \emph{half-space separation} problem: with an implicitly given convexity space $(V, \cs)$ and two (convex) subsets $A, B$ of $V$, are there complementary convex sets $H, \comp{H}$---the so-called \emph{half-spaces}---such that $A \subseteq H$ and $B \subseteq \comp{H}$?
This problem is a generalization to abstract convexity of the half-space separation problem in $\mathbb{R}^d$, being well-studied in machine learning \cite{boser1992training, freund1998large, vapnik1998statistical}.
Half-space separation has motivated the study of structural separation properties of convexity spaces.
Among these properties, two have received particular attention, notably within graph convexities (see e.g.~\cite{bandelt1989graphs, chepoi1994separation, ellis1952general, kay1971axiomatic, van1984binary}): the $S_3$ property stating that any convex set $C$ can be separated from any element of $V$ not in $C$; and the $S_4$ or Kakutani property stating that any pair of disjoint convex sets can be separated.
Besides, the study of the half-space separation problem on its own has recently been brought to the context of graph convexities \cite{seiffarth2020maximal, thiessen2021active}.
In particular, Seiffarth et al.~\cite{seiffarth2020maximal} show that half-space separation is $\NP$-complete for geodesic convexity, the convexity induced by the shortest paths of a graph.
To our knowledge though, the complexity status of half-space separation for the other aforementioned graph convexities is still unknown.

In our contribution, we follow this latter line of research and study half-space separation for the monophonic convexity.
Given a graph $G$ with vertices $V(G)$, a set $C \subseteq V(G)$ is monophonically convex if for any two vertices $u, v$ of $C$, all the vertices on a chordless path $u$ and $v$ lie in $C$.
We prove that half-space separation can be decided in polynomial time for monophonic convexity.
More formally, our main theorem reads as follows:

\begin{theorem}[restate=THMseppoly, label=thm:separability-polynomial] 
Given a graph $G$ and two subsets $A, B$ of $V(G)$, whether $A$, $B$ are separated by monophonic half-spaces can be decided in polynomial time.
\end{theorem}

Theorem \ref{thm:separability-polynomial} contrasts with the $\NP$-completeness of half-space separation for geodesic convexity \cite{seiffarth2020maximal} and suggests to study separation in further graph convexities.
Besides, half-space separation also relates to the $p$-partition problem (in graph convexities).
In the $p$-partition problem, one is given a graph $G$ and has to decide whether $V(G)$ can be partitioned into $p$ convex sets, where the meaning of convex depends on the convexity at hand.
For monophonic convexity, Gonzalez et al.~\cite{gonzalez2020covering} show that $p$-partition is $\NP$-complete for $p \geq 3$, but they leave open the case $p = 2$.
Since $2$-partition is possible if and only if there exists two separable vertices, Theorem \ref{thm:separability-polynomial} proves that $2$-partition can be decided in polynomial time.

\begin{remark}
In very recent contributions, Chepoi \cite{chepoi2024separation} and Bressan et al.~\cite{bressan2024efficient} also showed that half-space separation can be decided in polynomial time for monophonic convexity. 
Their results have been obtained independently and using different approaches, even though they share some common points with the technique used in this paper.
\end{remark}

\userpar{Organization of the paper.} 
In Section \ref{sec:preliminaries} we provide definitions, notations and we formally define the problem we investigate in the paper.
In Section \ref{sec:algorithm} we prove Theorem \ref{thm:separability-polynomial} by giving an algorithm which decides whether two sets can be separated by half-spaces.
We conclude the paper in Section \ref{sec:conclusion}.

\section{Preliminaries} \label{sec:preliminaries}

All the objects considered in this paper are finite.
Let $V$ be a set. 
The powerset of $V$ is denoted $\pow{V}$.
Given $X \subseteq V$, we write $\comp{X}$ the complement of $X$ in $V$, i.e., $\comp{X} = V \setminus X$.
Sometimes, we shall write a set $X$ as the concatenation of its elements, e.g., $uv$ instead of $\{u, v\}$.
As a result, $X \cup uv$ and $X \cup v$ stands for $X \cup \{u, v\}$ and $X \cup \{v\}$ respectively.

\userpar{Graphs.}
We assume the reader is familiar with standard graph terminology.
We consider loopless undirected graphs.
Let $G$ be a graph with vertices $V(G)$ and edge set $E(G)$.
A subgraph of $G$ is any graph $H$ such that $V(H) \subseteq V(G)$ and $E(H) \subseteq E(G)$.
The \emph{(open) neighborhood} of a vertex $v$ in $G$ is denoted $N(v)$ and is defined as $N(v) = \{u \in V(G) \st uv \in E(G)\}$.
The \emph{closed neighborhood} of $v$ in $G$ is $N[v] = N(v) \cup v$.
Let $X \subseteq V(G)$.
For $X \subseteq V(G)$, we put similarly $N(X) = \{u \in V(G) \setminus X \st xu \in E(G) \text{ for some } x \in X\}$ and $N[X] = N(X) \cup X$.
The subgraph of $G$ \emph{induced} by $X$ is $G[X] = (X, E(G[X]))$, where $E(G[X]) = \{ uv \in E(G) \st u, v \in X\}$.
If this is clear from the context, we identify $X$ with $G[X]$, and we use $G - X$ to denote $G[V(G) - X]$. 
A \emph{path} in $G$ is a subgraph $P$ of $G$ with $V(P) = \{v_1, \dots, v_k\}$ and such that $v_i v_{i + 1} \in E(P)$ for each $1 \leq i < k$.
Putting $u = v_1$ and $v = v_k$, $P$ is an \emph{$uv$-path} of $G$.
An \emph{induced $uv$-path} or \emph{chordless $uv$-path} of $G$ is an induced subgraph of $G$ being an $uv$-path.
A shortest path is an induced path with the least possible number of vertices.
For simplicity we will identify a path $P$ paths with the sequence $v_1, \dots, v_k$ of its vertices.
Let $A, B \subseteq V(G)$ be non-empty.
The \emph{(inner) frontier} of $A$ with respect to $B$ is $F(A, B) = A \cap N[B]$.
We note that if $A, B$ are disjoint, we obtain $F(A, B) = A \cap N(B)$. 
Remark that for every $X \subseteq V(G)$, $F(\comp{X}, X) = N(X)$.

\userpar{Convexity spaces.}
We refer the reader to \cite{van1993theory} for a thorough introduction to convexity theory.
A \emph{convexity space} is a pair $(V, \cs)$, with $\cs \subseteq \pow{V}$, such that $\emptyset, V \in \cs$ and for every $C_1, C_2 \in \cs$, $C_1 \cap C_2 \in \cs$.
The sets in $\cs$ are \emph{convex} sets.
A convexity space $(V, \cs)$ induces a \emph{(convex) hull} operator $\cl \colon \pow{V} \to \pow{V}$ defined for all $X \subseteq V$ by:
\[ 
\cl(X) = \bigcap \{C \in \cs \st X \subseteq C\}
\]
The operator $\cl$ satisfies, for all $X, Y \subseteq V$: $X \subseteq \cl(X)$; $ \cl(X) \subseteq \cl(Y)$ if $X \subseteq Y$; and $\cl(\cl(X)) = \cl(X)$.
The \emph{Carathéodory number} of $(V, \cs)$ is the least integer $d$ such that for every $X \subseteq V$ and $v \in V$, if $v \in \cl(X)$, there exists a subset $Y$ of $X$ such that $v \in \cl(Y)$ and $\card{Y} \leq d$.
A \emph{half-space} of $(V, \cs)$ is a convex set $H$ whose set complement $\comp{H}$ is convex, that is, $H, \comp{H} \in \cs$.
Let $A, B$ be two subsets of $V$.
We say that $A$ and $B$ are \emph{(half-space) separable} if there exists half-spaces $H, \comp{H}$ satisfying $A \subseteq H$ and $B \subseteq \comp{H}$.
This is equivalent to $\cl(A) \subseteq H$ and $\cl(B) \subseteq \comp{H}$.
The \emph{shadow of $A$ with respect to $B$} \cite{chepoi1986some, chepoi1994separation} is the set $A / B = \{v \in V \st \cl(B \cup v) \cap A \neq \emptyset\}$.
Observe that $A \subseteq A / B$ and $B \subseteq B / A$.

\begin{remark}
Usually, $A / B$ is defined for disjoint sets.
Here, it is more convenient to extend this definition to sets that may intersect.
If $A \cap B \neq \emptyset$, then $A / B = V$ vacuously.
\end{remark}

\userpar{Monophonic convexity.}
We introduce monophonic convexity.
We redirect the reader to \cite{van1993theory, pelayo2013geodesic} for further details on graph and interval convexities.
Let $G$ be a graph, and let $u, v \in V(G)$.
The \emph{monophonic closed interval} of $u, v$ is the set of all vertices that lie on a chordless $uv$-path, denoted by $J[u, v]$.
For $X \subseteq V(G)$, we put $J[X] = \bigcup_{u, v \in X} J[u, v]$.
A set $C$ is \emph{monophonically convex} if $J[C] = C$.
Throughout the paper, if there is no ambiguity, we use the term convex sets as a shortening for monophonically convex sets. 
With $\cs = \{C \subseteq V(G) \st C \text{ is monophonically convex}\}$, the pair $(V(G), \cs)$ is a convexity space, the \emph{monophonic convexity} of $G$.
Its convex hull operator $\cl$ is defined for all $X \subseteq V(G)$ by:
\[ 
\cl(X) = \bigcup_{k = 0}^{\infty} X_k
\]
where $X_0 = X$ and $X_i = J[X_{i - 1}]$ for $i \geq 1$.
We now gather existing results regarding monophonic convexity that will be useful throughout the paper.

\begin{theorem}[\cite{dourado2010complexity}, Theorem 2.1] \label{thm:dourado-convex}
Let $G$ be a connected graph and let $C \subseteq V$.
The set $C$ is convex if and only if for every connected component $S$ of $G - C$, $F_G(C, S)$ is a clique.
\end{theorem}

\begin{lemma}[\cite{gonzalez2020covering}, Lemma 14] \label{lem:gonzales-separator}
Let $G$ be a connected graph, $K$ a clique separator of $G$, and $X$ the union of some of the connected components of $G - K$.
Then $X \cup K$ is convex.
\end{lemma}

\begin{observation}[see also \cite{duchet1988convex}] \label{obs:connected}
In a connected graph $G$, every convex set is connected.
\end{observation}

\begin{theorem}[\cite{duchet1988convex}, Theorem 5.1] \label{thm:duchet-cara}
The monophonic convexity of a connected graph has Carathéodory number is $1$ if the graph is a clique and $2$ otherwise.
\end{theorem}

\begin{theorem}[\cite{dourado2010complexity}, Theorem 4.1] \label{thm:dourado-poly}
Let $G$ be a graph and let $X \subseteq V(G)$.
Then, $\cl(X)$ can be computed in polynomial time in the size of $G$.
\end{theorem}

We end these preliminaries by stating the problem we investigate in this paper.
It is the problem of separating two sets of vertices by half-spaces.
Its decision version is:
\begin{decproblem}
	\problemtitle{Half-space separation in monophonic convexity}
	\probleminput{A graph $G$ and two (non-empty and disjoint) subsets $A, B$ of $V(G)$.}
	\problemquestion{Are $A$ and $B$ half-space separable?} 
\end{decproblem}
Since $\cl$ can be computed in polynomial time by Theorem \ref{thm:dourado-poly} and $A, B$ are separable if and only if $\cl(A)$, $\cl(B)$ are separable, we can assume w.l.o.g.~that the sets $A$ and $B$ are convex.

\section{Half-space separation}
\label{sec:algorithm}

In this section, we prove Theorem \ref{thm:separability-polynomial}, which we first restate.

\THMseppoly*

Remark that if the input graph is not connected, one just has to solve the problem for each connected component.
Thus, we can consider without loss of generality that the graphs we consider are connected.
Hence, for the section, we fix a connected graph $G$.
Let $A, B$ be two (disjoint) convex sets of $G$.
To prove Theorem \ref{thm:separability-polynomial}, we give a polynomial time algorithm that decides whether $A$, $B$ are separable.
The algorithm first computes a shortest path $a = v_1, \dots, v_k = b$ for some $a \in A$ and $b \in B$ in polynomial time.
We show in Lemma \ref{lem:linked} that $A$ and $B$ are separable if and only if there exists $1 \leq i < k$ such that $ A_i := \cl(A \cup \{v_1, \dots, v_i\})$ and $B_i := \cl(B \cup \{v_{i+1}, \dots, v_k\})$ are separable.
This step is the \emph{linkage} of $A$ and $B$.
Then, for each $i$, the algorithm does the subsequent operations:
\begin{enumerate}[(1)]
    \item It computes the \emph{saturation} $A'_i := S(A_i, B_i)$, $B'_i := S(B_i, A_i)$ of $A_i, B_i$ (resp.) with respect to $\cl$ (see Subsection \ref{subsec:saturation}).
    Informally, the saturation step extends $A_i$ and $B_i$ with vertices that are forced on one of the two sides by the hull operator $\cl$.
    Lemma \ref{lem:saturation} shows that $A_i$, $B_i$ are separable if and only if $A_i'$, $B'_i$ are separable.
    Corollary \ref{cor:saturation-polynomial} proves that computing saturation takes polynomial time.
    \item From $A_i'$ and $B_i'$, it builds an equivalence relation $\equiv_{A'_i B'_i}$ on $\comp{A'_i \cup B'_i}$ and an associated graph $G_{A'_i B'_i}$.
    Theorem \ref{thm:linked-separability} states that $A'_i$ and $B'_i$ are separable if and only if $G_{A'_i B'_i}$ is bipartite and no equivalence class of $\equiv_{A'_i B'_i}$ contains a so-called \emph{forbidden pair} of vertices.
    Finally, Theorem \ref{thm:linked-separability-poly} proves that the conditions of Theorem \ref{thm:linked-separability} can be tested in polynomial time.
\end{enumerate}
The algorithm output that $A$ and $B$ are separable if there is an integer $i$ for which step (2) succeeds.
Otherwise, $A$ and $B$ are not separable.
The correctness of the algorithm follows from Lemma \ref{lem:linked}, Lemma \ref{lem:saturation} and Theorem \ref{thm:linked-separability}.
The fact that it runs in polynomial time is a consequence of Corollary \ref{cor:saturation-polynomial} and Theorem \ref{thm:linked-separability-poly}.
This proves Theorem \ref{thm:separability-polynomial}.

The rest of the section is dedicated to the proof of the aforementioned statements.

\subsection{Linkage along a shortest path}
\label{subsec:linkage}

Let $A, B$ be two non-empty disjoint convex subsets of $V(G)$.
Assume that $A$ and $B$ are separable and let $H, \comp{H}$ be half-spaces separating $A$ and $B$.
Then for each $a \in A$, and each $b \in B$, all the vertices on the shortests $ab$-paths are distributed among $H$ and $\comp{H}$.
We show that, in fact, for each shortest $ab$-path, there is a vertex before which all vertices are assigned one half-space and all vertices after which are assigned the other half-space.

\begin{proposition}[restate=PROPabpath, label=prop:ab-path] 
Let $a \in A$, $b \in B$ and let $a = v_1, \dots, v_k = b$ be a shortest $ab$-path.
For every half-space separation $H, \comp{H}$ of $A$ and $B$ with $A \subseteq H$ and $B \subseteq \comp{H}$, there exists $1 \leq i < k$ such that $\{v_1, \dots, v_i\} \subseteq H$ and $\{v_{i+1}, \dots, v_k\} \subseteq \comp{H}$.
\end{proposition}

\begin{proof}
    Assume for contradiction there exists half-spaces $H$, $\comp{H}$ with $A \subseteq H$, $B \subseteq \comp{H}$, and such that for every $1 \leq i < k$, either $\{v_1, \dots, v_i\} \nsubseteq H$ or $\{v_{i+1}, \dots, v_k\} \nsubseteq \comp{H}$.
    Observe that $k \geq 4$ must hold.
    Consider the case $i = 1$ so that $a = v_1$ and $v_1 \in H$ by definition of $H$.
    By assumption, there exists $i < j < k$ such that $v_j \in H$.
    Consider the largest such $j$.
    Then, $\{v_1, \dots, v_j\} \nsubseteq H$ must also hold.
    Thus, there exists $1 < \ell < j$ such that $v_{\ell} \in \comp{H}$.
    As $a = v_1, \dots, v_k = b$ is a shortest $ab$-path, it is chordless and $a = v_1, \dots, v_j$ is a chordless $av_j$-path.
    We deduce $v_{\ell} \in J[a, v_j] \subseteq H$ and $H \cap \comp{H} \neq \emptyset$, a contradiction with $H, \bar H$ being half-spaces.
    This concludes the proof.
    \end{proof}

\begin{figure}[ht!]
\centering
\includegraphics[scale=\FIGLinkage]{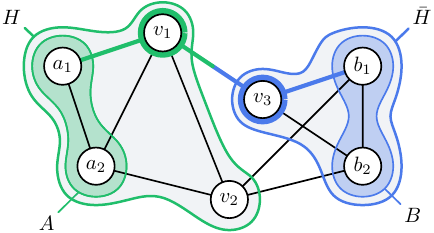}  
\caption{A graph $G$ with two disjoint convex sets $A= \{a_1, a_2\}$ and $B = \{b_1, b_2\}$ (circled in green and blue resp.).
$A$ and $B$ are not linked, but they can be linked along the path $a_1, v_1, v_3, b_1$ (in bold green / bold blue).
Namely, $A \cup v_1$ and $B \cup v_3$ are linked and convex.
Two half-spaces $H, \comp{H}$ separating $A \cup v_1$ and $B \cup v_3$ (hence $A$ and $B$) are drawn.}
\label{fig:linkage}
\end{figure}

Following Proposition \ref{prop:ab-path}, we say that $A$ and $B$ are \emph{linked} if there exists $a \in A$, $b \in B$ such that $ab \in E(G)$.
Linked sets and Proposition \ref{prop:ab-path} are illustrated in Figure \ref{fig:linkage}.
The next lemma is a direct consequence of Proposition \ref{prop:ab-path}.

\begin{lemma}[restate=LEMlinked, label=lem:linked] 
Let $G$ be a connected graph and let $A, B$ be two non-empty disjoint convex subsets of $V(G)$.
Let $a \in A$, $b \in B$ and let $a = v_1, \dots, v_k = b$ be a shortest $ab$-path.
Then, $A$ and $B$ are separable if and only if there exists $1 \leq i < k$ such that $\cl(A \cup \{v_1, \dots, v_i\})$ and $\cl(B \cup \{v_{i+1}, \dots, v_k\})$ are separable.
\end{lemma}

Given $a \in A$ and $b \in B$, finding a shortest $ab$-path can be done in polynomial time.
Hence, making $A$ and $B$ linked can be done efficiently.
Moreover, if $A$ and $B$ are linked, then for any disjoint $A', B' \subseteq V$ such that $A \subseteq A'$ and $B \subseteq B'$, $A'$ and $B'$ must be linked too.
In what follows, we will thus consider disjoint, convex and linked subsets of $V(G)$.

\subsection{Saturation with the hull operator}
\label{subsec:saturation}

Let $A, B$ be two disjoint, linked and convex subsets of $V(G)$.
In this part, we use the hull operator $\cl$ to define two sets $S(A, B)$ and $S(B, A)$---the saturation of $A$ and $B$ (see below)---with $A \subseteq S(A, B)$, $B \subseteq S(B, A)$ and such that $A, B$ are separable if and only if their saturation is separable.
Informally, we use $\cl$ to identify vertices that will appear in the same half-space as $A$ in any possible half-space separation of $A$ (and similarly with $B$), if any.
We use two properties built on $\cl$:
\begin{enumerate}[(1)] 
\item \textbf{Shadow-closing}.
Remind that $A / B$, the shadow of $A$ with respect to $B$, is defined by $A / B = \{v \in V(G) \st \cl(B \cup v) \cap A \neq \emptyset\}$.
In particular, $A \subseteq A / B$.

\item \textbf{Forbidden sets}.
Let $X \subseteq \comp{A \cup B}$ and assume that $\cl(X) \cap A \neq \emptyset$ and $\cl(X) \cap B \neq \emptyset$.
Since $\cl(v) = \{v\}$ for all $v \in V$, we have $\card{X} \geq 2$.
Thus, separating $A, B$ implies to split the vertices of $X$.
We say that $X$ is a \emph{forbidden set} of $A$ and $B$ with respect to $G$.
A set $X$ is forbidden if and only if it includes an inclusion-wise minimal forbidden set as a subset.
Henceforth, in order to use forbidden sets, we need only consider the family of inclusion-wise minimal forbidden sets, denoted $\mfs(A, B)$.
Formally,
\[ 
\mfs(A, B) = \min_{\subseteq} \{X \subseteq \comp{A \cup B} \st \cl(X) \cap A \neq \emptyset \text{ and } \cl(X) \cap B \neq \emptyset \}.
\]
\end{enumerate}
\begin{figure}[ht!]
\centering
\includegraphics[page=1, scale=\FIGPreSatOne]{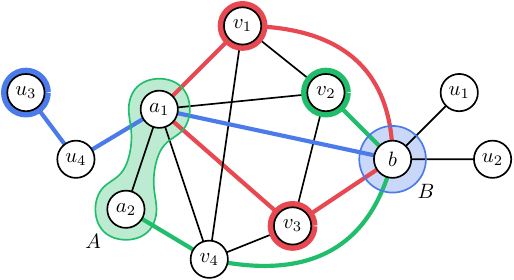}  
\caption{A graph $G$ in which we seek to separate $A$ and $B$ (in circled green/blue resp.).
The vertex $a_1$ is on a chordless $u_3b$-path (bold blue), so that $u_3 \in A / B$.
Dually, $b$ is on a chordless $v_2a_2$-path (bold green), i.e., $v_2 \in B / A$.
Besides, $\cl(v_1v_3)$ intersects both $A$ and $B$ (bold red).
Thus, $v_1, v_3$ must be separated to separate $A$ and $B$, and $v_1v_3 \in \mfs(A, B)$ holds.}
\label{fig:pre-saturation-1}
\end{figure}
We illustrate shadows and forbidden sets in Figure \ref{fig:pre-saturation-1}.
Now, we use $A / B$ (resp. $B / A$) and $\mfs(A, B)$ in view of separating $A$ and $B$.
On the one hand, $A / B$ cannot be separated from $A$ by definition.
On the other hand, for each $X \in \mfs(A, B)$ and every half-spaces $H, \bar H$ separating $A$ and $B$ with $A \subseteq H$, there exists at least one $x \in X$ such that $x \in H$, so that, $\bigcap_{x \in X} \cl(A \cup x) \subseteq H$ always hold.
Based on the previous arguments, we define the \emph{pre-saturation} of $A$ with respect to $B$ in $G$, denoted by $\sigma(A, B)$, by:%
\[ 
\sigma(A, B) = \cl\left(A / B \cup \bigcup \left\{ \bigcap_{x \in X} \cl(A \cup x) \st X \in \mfs(A, B)\right\}\right) 
\]
Observe that if $A \cap B \neq \emptyset$, then $\sigma(A, B) = \sigma(B, A) = V(G)$ as $A / B = B / A = V(G)$.
In this case though, $A$ and $B$ cannot be separated.
We prove in the next statement that $\sigma(A, B)$ preserves separation.
Remark that it holds regardless of the disjointness of $A$ and $B$.

\begin{lemma}[restate=LEMPreSaturation, label=lem:pre-saturation] 
Let $G$ be a connected graph, and let $A, B$ be linked and convex subsets of $V(G)$.
Then, $A, B$ are separable if and only if $\sigma(A, B)$ and $\sigma(B, A)$ are separable.
\end{lemma}

\begin{proof}
The if part follows from $A \subseteq A / B \subseteq \sigma(A, B)$ and $B \subseteq B / A \subseteq \sigma(B, A)$.
We show the only if part. 
Suppose that $A$ and $B$ are separable and let $H, \comp{H}$ be half-spaces such that $A \subseteq H$, $B \subseteq \bar H$.
Let $v \in A / B$.
By definition, $\cl(B \cup v) \cap A \neq \emptyset$, hence $H \cap \comp{H} = \emptyset$ entails $v \in H$.
Now let $X \in \mfs(A, B)$.
By definition of forbidden sets, $X \cap H \neq \emptyset$ and $X \nsubseteq H$.
Thus, there exists $x \in X$ such that $x \in H$, which entails $\cl(A \cup x) \subseteq H$ as $H$ is convex.
Since $\bigcap_{x' \in X} \cl(A \cup x') \subseteq \cl(A \cup x)$ for each $x \in X$, we deduce
\[ A / B \cup \bigcup \left\{ \bigcap_{x \in X} \cl(A \cup x) \st X \in \mfs(A, B)\right\} \subseteq H.\]
As $H$ is convex, we get $\sigma(A, B) \subseteq H$.
Applying the symmetric reasoning on $\sigma(B, A)$ yields $\sigma(B, A) \subseteq \comp{H}$.
This concludes the proof.
\end{proof}

\begin{figure}[ht!]
\centering
\includegraphics[scale=\FIGPreSatTwo, page=2]{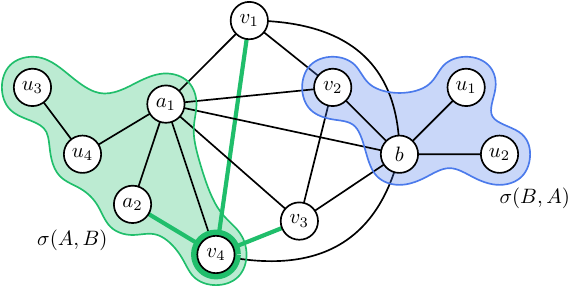}
\caption{Pre-saturation applied to the sets $A$ and $B$ of Figure \ref{fig:pre-saturation-1}.
For $\sigma(B, A)$, $u_1, u_2 \in B / A$ are added.
For $\sigma(A, B)$, we have $u_3, u_4 \in A / B$ and $v_4 \in \cl(A \cup v_1) \cap \cl(A \cup v_3)$ (paths in bold green) with $v_1v_3 \in \mfs(A, B)$.}
\label{fig:pre-saturation-2}
\end{figure}
We illustrate pre-saturation in Figure \ref{fig:pre-saturation-2}, where the operation is applied to the sets $A$ and $B$ of Figure \ref{fig:pre-saturation-1}.
In this example, once pre-saturation has been applied, no further vertices can be assigned by applying pre-saturation once more.
There are cases however where applying pre-saturation twice yields new vertices to assign.
Figure \ref{fig:saturation} illustrates this situation.
\begin{figure}[ht!]
\centering
\includegraphics[scale=\FIGSaturation, page=3]{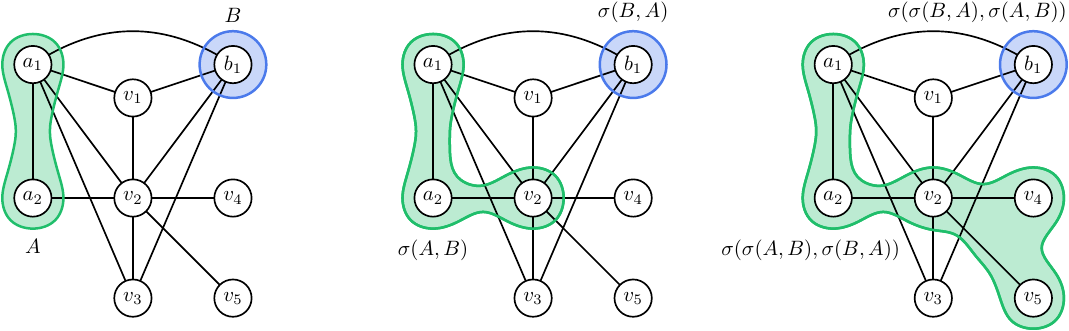}
\caption{An example where pre-saturation can be applied twice.
For $\sigma(A, B)$, $v_2$ is obtained from the forbidden pair $v_1v_3$.
Once $v_2$ is added, $v_4, v_5$ become part of $\sigma(A, B) / \sigma(B, A)$.
Observe that $B = \sigma(B, A) = \sigma(\sigma(B, A), \sigma(A, B))$.
The remaining vertices $v_1, v_3$ can be separated in any way.}
\label{fig:saturation}
\end{figure}
This suggests to iteratively apply the pre-saturation operator until no more vertices are added.
For $A, B \subseteq V$, the \emph{saturation} of $A$ with respect to $B$, denoted by $S(A, B)$ is defined as follows:
\[ 
S(A, B) = \bigcup_{i = 0}^{\infty} \sigma(A_i, B_i)
\]
where $A_0 = A$, $B_0 = B$ and for all $1 \leq i$, $A_i = \sigma(A_{i-1}, B_{i-1})$ and $B_i = \sigma(B_{i - 1}, A_{i - 1})$.
Given $A, B \subseteq V(G)$, we say that $A$ and $B$ are \emph{saturated} if $A = S(A, B)$ and $B = S(B, A)$.
Since $\sigma$ is increasing, the procedure for computing $S(A, B)$ terminates after $\card{V(G)}$ steps at most.
Applying Lemma \ref{lem:pre-saturation} inductively on $1 \leq i$, we get:

\begin{lemma}[restate=LEMSaturation] \label{lem:saturation}
Let $G$ be a connected graph, and let $A, B$ be two linked and convex subsets of $V(G)$.
Then, $A, B$ are separable if and only if $S(A, B)$, $S(B, A)$ are separable.
\end{lemma}

\begin{remark}
If $A_i \cap B_i \neq \emptyset$ for some $i$, then $S(A, B) = S(B, A) = V(G)$, and no separation can distinguish $S(A, B)$ and $S(B, A)$.
In particular, $A, B$ are thus not separable.
\end{remark}

To conclude this paragraph, we argue that $S(A, B)$ can be computed in polynomial time in the size of $G$.
Since $S$ is at most $\card{V(G)}$ applications of $\sigma$ on subsets of $V(G)$, it is sufficient to show that $\sigma$ can be computed in polynomial time.
The bottleneck of computing $\sigma$ lies in finding $\mfs(A, B)$.
However, the fact that the Carathéodory number of monophonic convexity is $2$ by Theorem \ref{thm:duchet-cara} makes the sets in $\mfs(A, B)$ of constant size.

\begin{proposition}[restate=PROPSatPolynomial, label=prop:saturation-polynomial] 
Given $A, B \subseteq V(G)$, $\sigma(A, B)$ can be computed in polynomial time in the size of $G$.
\end{proposition}

\begin{proof}
Recall that $\cl(A)$ can be computed in polynomial time for every $A \subseteq V$ by Theorem \ref{thm:dourado-poly}.
Henceforth, $A / B$ can be computed in polynomial time by checking whether $\cl(B \cup v) \cap A \neq \emptyset$ for each $v \notin A \cup B$.

We move to the computation of $\mfs(A, B)$. 
Let $X$ be a forbidden set of $A, B$.
It follows that $G$ is not a clique.
Hence, since the Carathéodory number of the monophonic convexity is 2 by Theorem \ref{thm:duchet-cara}, we deduce that there exists $u_1, v_1, u_2, v_2 \in X$ and $a \in A$, $b \in B$ such that $a \in \cl(u_1 v_1)$ and $b \in \cl(u_2 v_2)$, possibly with $u_1 = u_2$ or $v_1 = v_2$.
Hence, if $X$ is chosen minimal among forbidden sets, i.e., $X \in \mfs(A, B)$, $X$ precisely consists in the elements $u_1, v_1, u_2, v_2$. 
Thus, $\card{X} \leq 4$.
Therefore, $\mfs(A, B)$ can be computed in polynomial time by checking $\cl(X)$ for all subsets $X$ of $V(G)$ of size at most $4$.
Computing $\bigcap_{x \in X} \cl(A \cup x)$ for some $X \in \mfs(A, B)$ thus requires a constant number of calls to $\cl$.
We conclude that $\sigma(A, B)$ can be computed in polynomial time as required.
\end{proof}

\begin{corollary}[restate=CORSatPolynomial] \label{cor:saturation-polynomial}
Given $A, B \subseteq V(G)$, $S(A, B)$ can be computed in polynomial time in the size of $G$.
\end{corollary}

We note that saturation is not sufficient to decide separability, as suggested by Figure \ref{fig:saturation-not-sufficient}.
This motivates the last step of the algorithm. 
\begin{figure}[ht!]
\centering
\includegraphics[scale=\FIGSatNotSuf, page=4]{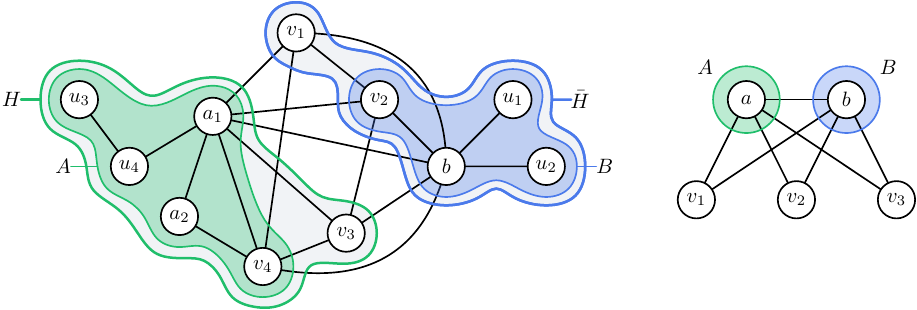}
\caption{Two cases where $A$ and $B$ are linked, convex and saturated.
On the left (follow-up of Figure \ref{fig:pre-saturation-2}), $A$ and $B$ can be separated (two half-spaces are drawn).
On the right, any bipartition of the vertices will contain one of the forbidden pair $v_1v_2, v_1v_3$ or $v_2v_3$.
Thus, $A$ and $B$ are not separable.}
\label{fig:saturation-not-sufficient}
\end{figure}

\subsection{Testing bipartiteness}
\label{subsec:decision}

Let $A, B$ be two linked, disjoint and saturated subsets of $V(G)$.
By definition of saturation, $A$ and $B$ are convex.
We characterize the separability of $A$ and $B$ using an equivalence relation $\eq$ on $\comp{A \cup B}$ and a graph $G_{AB}$ defined from $\eq$.
More precisely, we prove in Theorem \ref{thm:linked-separability} that $A$ and $B$ are separable if and only if $G_{AB}$ is bipartite and no two $\eq$-equivalent vertices form a forbidden pair of $\mfs(A, B)$.

As a preliminary step though, we give properties of $G$ and $N(A \cup B)$ in terms of $A$ and $B$.
We start with a statement that holds for every convex set.

\begin{proposition}[restate=PROPConvex, label=prop:convex] 
Let $C \subseteq V(G)$ be a convex set, and let $u, v$ be two distinct vertices of $V(G) \setminus C$.
Then:
\begin{enumerate}[(1)]
    \item if $u, v$ are not adjacent, then $\cl(uv) \cap C \neq \emptyset$ if and only if there exists $u', v' \in N(C)$ such that $u'v' \notin E(G)$ and $u', v' \in \cl(uv)$;
    \item if $u, v$ are adjacent, then $F(C, u) \setminus F(C, v) \neq \emptyset$ entails $u \in \cl(C \cup v)$.
\end{enumerate}
\end{proposition}

\begin{proof}
    We prove item (1).
    We start with the if part.
    Let $u, v$ be non-adjacent and assume there exists $u', v' \in N(C)$ such that $u'v' \notin E(G)$ and $u', v' \in \cl(uv)$.
    Since $C$ is convex and $G$ is connected by assumption, $C$ is convex by Observation \ref{obs:connected}.
    Therefore, $C \cup u'v'$ is connected and contains a chordless $u'v'$-path.
    It follows $C \cap \cl(uv) \neq \emptyset$.
    We move to the only if part.
    Assume that $\cl(uv) \cap C \neq \emptyset$ and consider $\cl(uv) \setminus C$.
    We have that $\cl(uv) \setminus C$ is not convex, so there exists $u', v' \in \cl(uv) \setminus C$ such that $J[u', v'] \nsubseteq \cl(uv) \setminus C$ but $J[u', v'] \subseteq \cl(uv)$.
    Thus there exists a chordless $u'v'$-path $u' = v_1, \dots, v_k = v'$ and some $1 \leq i < i + 1 < j \leq k$ such that $v_i, v_j \notin C$ and $v_\ell \in C$ for each $i < \ell < j$.
    Without loss of generality, we can assume $u' = v_i$ and $v' = v_j$ so that $u', v' \in N(C)$.
    Since $v_1, \dots, v_k$ is chordless and goes through $C$, we deduce that $u', v'$ are not adjacent.
    This concludes the proof of item (1).
    
    We move to item (2).
    Assume that $uv \in E(G)$.
    Let $w \in F(C, u) \setminus F(C, v)$.
    Then, $w, u, v$ is a chordless $wv$-path.
    Hence, $u \in \cl(C \cup v)$ holds as required.
    \end{proof}

Leveraging from the fact that $A, B$ are linked and saturated, we use Proposition \ref{prop:convex} to show that every vertex in $N(A \cup B)$ is adjacent to both $A$ and $B$.

\begin{lemma}[restate=LEMNeighborsEqual, label=lem:neighbors-equal] 
Let $G$ be a connected graph and let $A, B$ be two linked, disjoint and saturated subsets of $V(G)$.
For every $v \in N(A \cup B)$, $F(A, B) \cup F(B, A) \subseteq N(v)$.
Therefore, the following properties hold for $A$ (and symmetrically for $B$):
\begin{enumerate}[(1)]
    \item $N(A) = F(B, A) \cup N(A \cup B)$;
    \item $F(A, \comp{A}) = F(A, N(A \cup B))$ is a clique.
\end{enumerate}
\end{lemma}

\begin{proof}
    Assume for contradiction there exists $v \in N(A \cup B)$ such that $F(A, B) \cup F(B, A) \nsubseteq N(v)$.
    We have two cases: 
    \begin{enumerate}[(1)]
        \item $F(A, B) \nsubseteq N(v)$ and $F(B, A) \nsubseteq N(v)$.
        Suppose w.l.o.g.~that $v \in N(A)$.
        There exists $b \in F(B, A)$ such that $b \notin N(v)$.
        Then, we deduce by Proposition \ref{prop:convex} that $\cl(bv) \cap A \neq \emptyset$ and $v \in A / B$.
        
        \item $F(A, B) \subseteq N(v)$ and $F(B, A) \nsubseteq N(v)$ (w.l.o.g.).
        Since $F(A, B) \subseteq N(v)$, $v \in N(A)$ holds.
        Thus, $v \in A / B$ again follows from Proposition \ref{prop:convex}.
    \end{enumerate}
    In both cases, we obtain $v \in A / B$ with $v \notin A$.
    This contradicts $A$ being saturated.
    We derive $F(A, B) \cup F(B, A) \subseteq N(v)$.
    Therefore, every $v \in N(A) \setminus B$ also lies in $N(B) \setminus A$ so that $N(A) \cap N(B) = N(A \cup B)$ holds along with $N(A) = F(B, A) \cup N(A \cup B)$ and $F(A, \comp{A}) = F(A, N(A \cup B))$.
    To see that $F(A, \comp{A})$ is a clique, observe that $B \cup N(A \cup B)$ is connected since $B$ is convex.
    We deduce that $B \cup N(A \cup B)$ is included in a connected component of $G - A$.
    Since $F(A, \comp{A}) = F(A, N(A \cup B))$ and $A$ is convex as it saturated, we obtain from Theorem \ref{thm:dourado-convex} that $F(A, \comp{A})$ is a clique.
    \end{proof}

Proposition \ref{prop:convex} and Lemma \ref{lem:neighbors-equal} have two consequences.
First, we can characterize $\mfs(A, B)$ as the set of pairs $uv$ the closure of which contains non-adjacent vertices of $N(A \cup B)$, or in other words, a forbidden pair within $N(A \cup B)$.

\begin{lemma}[restate=LEMMFS, label=lem:MFS] 
Let $G$ be a connected graph and let $A, B$ be two linked, disjoint and saturated subsets of $V(G)$.
The following equality holds:
\[ 
\mfs(A, B) = \{uv \subseteq \comp{A \cup B} \st \cl(uv) \cap N(A \cup B) \text{ is not a clique} \}
\]
In particular, $X \subseteq \comp{A \cup B}$ is forbidden if and only if it includes a forbidden pair of $\mfs(A, B)$.
\end{lemma}

\begin{proof}
We show double inclusion and start with the $\supseteq$ part.
Remind that $\cl(v) = \{v\}$ for all $v \in V$.
Now let $u, v \notin A \cup B$ such that $\cl(uv) \cap N(A \cup B)$ is not a clique.
By Proposition \ref{lem:neighbors-equal}, we have $N(A \cup B) \subseteq N(A)$ so that $\cl(uv) \cap A \neq \emptyset$ by Proposition \ref{prop:convex}.
Similarly, we obtain $\cl(uv) \cap B \neq \emptyset$.
Hence $uv \in \mfs(A, B)$.

We proceed to the $\subseteq$ part.
Let $X \in \mfs(A, B)$.
By definition of $\mfs(A, B)$ and by Proposition \ref{prop:convex}, there exists $u, v \in X$ such that $\cl(uv)$ contains non-adjacent vertices $u', v'$ of $N(A)$.
Since $B$ is convex and by Lemma \ref{lem:neighbors-equal}, we have $u', v' \notin F(B, A)$.
In other words, $u', v' \in N(B)$ holds too.
We deduce $\cl(u', v') \cap B \neq \emptyset$ by Proposition \ref{prop:convex}.
Henceforth, $\cl(uv) \cap A \neq \emptyset$ and $\cl(uv) \cap B \neq \emptyset$.
It follows $X = uv$.

For the last part, a set is forbidden if and only if it includes an inclusion-wise minimal forbidden set.
Since $\mfs(A, B)$ contains only pairs of vertices, this concludes the proof.
\end{proof}

As another consequence, we can describe $N(A \cup B)$ and its interactions with $A$ and $B$ depending on whether it is a clique or not.

\begin{lemma}[restate=LEMAlternative, label=lem:alternative] 
Let $G$ be a connected graph and let $A, B$ be two linked, disjoint and saturated subsets of $V(G)$.
Then either $N(A \cup B)$ is a clique or for every $u, v \in N(A \cup B)$, $F(A, v) = F(A, u)$ and $F(B, v) = F(B, u)$.
\end{lemma}

\begin{proof}
Suppose that $N(A \cup B)$ is not a clique and let $u, v$ be two non-adjacent vertices of $N(A \cup B)$.
We first prove that $F(A, v) = F(A, u)$ and $F(B, v) = F(B, u)$.
Assume for contradiction that $F(A, v) \neq F(A, u)$.
We have $F(A, v) \setminus F(A, u) \neq \emptyset$ or $F(A, u) \setminus F(A, v) \neq \emptyset$.
By Proposition \ref{prop:convex}, we deduce $v \in \cl(A \cup u)$ or $u \in \cl(A \cup v)$.
Since $u, v$ are not adjacent, $uv \in \mfs(A, B)$ by Lemma \ref{lem:MFS} and we obtain $\cl(A \cup u) \cap B \neq \emptyset$ or $\cl(A \cup v) \cap B \neq \emptyset$.
Thus, either $u \in A /B$ or $v \in A / B$.
This contradicts $A$ being saturated.
We obtain $F(A, v) = F(A, u)$, and $F(B, v) = F(B, u)$ using the same argument on $B$.

Now, let $w \in N(A \cup B)$ such that $w \neq u, v$.
If $w$ is not adjacent to $u$ or $v$, then $F(A, w) = F(A, u) = F(A, v)$ and $F(B, w) = F(B, u) = F(B, v)$ readily holds by previous argument. 
Therefore, suppose that $w$ is adjacent to both $u$ and $v$.
We prove that: (1) $F(A, w) \setminus F(A, u) = \emptyset$ and (2) $F(A, u) \setminus F(A, w) = \emptyset$.
\begin{enumerate}[(1)]
\item Assume for contradiction that $F(A, w) \setminus F(A, u) \neq \emptyset$.
Then, $w \in \cl(A \cup u)$ by Proposition \ref{prop:convex}.
But since, $F(A, u) = F(A, v)$, we deduce $F(A, w) \setminus F(A, v) \neq \emptyset$ and hence $w \in \cl(A \cup v)$.
Because $uv \in \mfs(A, B)$ and $w \in \cl(A \cup u) \cap \cl(A \cup v)$, $w \notin A$ is a contradiction with $A$ being saturated.
We deduce that $F(A, w) \setminus F(A, u) = \emptyset$ must hold.

\item Again, suppose for contradiction that $F(A, u) \setminus F(A, w) \neq \emptyset$.
By Proposition 5, we obtain $u \in \cl(A \cup w)$ and since $F(A, u) = F(A, v)$, $v \in \cl(A \cup w)$ also holds.
Since $uv \in \mfs(A, B)$, we obtain $w \in B / A$, a contradiction with $B$ being saturated.
\end{enumerate}
We conclude that $F(A, w) = F(A, u)$ holds, and similarly $F(B, w) = F(B, u)$.
This concludes the proof.
\end{proof}

The two situations obtained from Lemma \ref{lem:alternative} are illustrated in Figure \ref{fig:alternative}.
\begin{figure}[ht!]
\centering
\includegraphics[scale=\FIGAlternative]{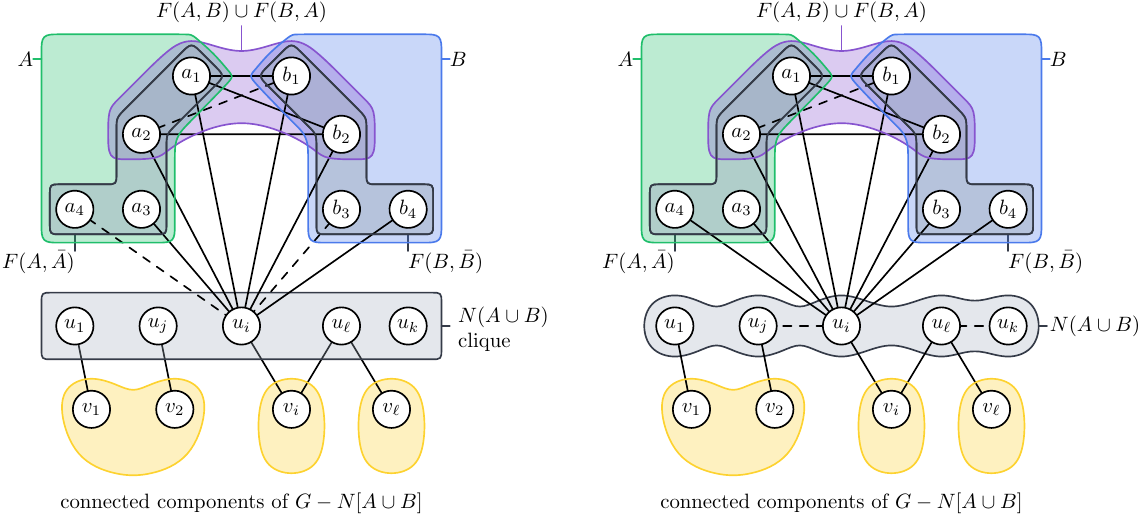}
\caption{The two possible situations of Lemma \ref{lem:alternative}.
On the left, $N(A \cup B)$ is a clique. Each vertex of $N(A \cup B)$, is connected to each vertex of $F(A, B) \cup F(B, A)$ (circled in purple), modelled by $u_i$. However, it needs not be adjacent to all the vertices of the cliques $F(A, \bar A)$ and $F(B, \bar B)$ (the dotted line $u_i a_4$ indicates a non-edge).
On the right, $N(A \cup B)$ is not a clique (for instance, $u_i, u_j$ are not adjacent).
Each vertex of $N(A \cup B)$ is complete to $F(A, \bar A) \cup F(B, \bar B)$, including $F(A, B) \cup F(B, A)$.}
\label{fig:alternative}
\end{figure}
In the case where $N(A \cup B)$ is not a clique, Lemma \ref{lem:alternative} together with Lemma \ref{lem:neighbors-equal} yields the subsequent corollary that will be useful later on.

\begin{corollary}[restate=CORCentralClique, label=cor:central-clique] 
If $N(A \cup B)$ is not a clique, then for every clique $K \subseteq N(A \cup B)$, $F(A, \comp{A}) \cup K$ (resp.~$F(B, \comp{B}) \cup K$) is a clique. 
\end{corollary}

\begin{proof}
The fact that $F(A, \comp{A})$ is a clique comes from Lemma \ref{lem:neighbors-equal}.
We show that $F(A, \comp{A}) \subseteq N(v)$ for every $v \in N(A \cup B)$.
First, $F(A, B) \subseteq N(v)$ holds by Lemma \ref{lem:neighbors-equal}.
Then, for a given $u \in F(A, \comp{A}) \setminus F(A, B)$, there exists by definition some $w \in N(A) \setminus B = N(A \cup B)$ such that $v \in F(A, w)$.
Because $N(A \cup B)$ is not a clique by hypothesis, $F(A, w) = F(A, v)$ by Lemma \ref{lem:alternative}.
We deduce $F(A, \comp{A}) \subseteq N(v)$ for all $v \in N(A \cup B)$.
The fact that $N(A \cup B) \subseteq N(v)$ for each $v \in F(A, \comp{A})$ follows.
Henceforth, $F(A, \comp{A}) \cup K$ is a clique.
\end{proof}

Thanks to Lemmas \ref{lem:MFS} and \ref{lem:alternative}, we are in position to relate the separability of $A, B$ with (co)bipartiteness.
We first address the case where all the vertices left to assign lie in $N(A \cup B)$, i.e., when $\comp{A \cup B} = N(A \cup B)$.
Although restricted, this case gives some insights for the general one. 

If $N(A \cup B)$ is a clique, then $\mfs(A, B) = \emptyset$ by Lemma \ref{lem:MFS}.
Hence every bipartition $X, Y$ of $N(A \cup B)$ readily satisfies $\cl(A \cup X) \cap B = \emptyset$ and $\cl(B \cup Y) \cap A = \emptyset$.
Therefore, $X$ and $Y$ need only satisfy $\cl(A \cup X) \cap Y = \emptyset$ and $\cl(B \cup X) \cap Y = \emptyset$.
The trivial bipartition $X = \emptyset$ and $Y = N(A \cup B)$ vacuously obeys this requirement.

\begin{remark}
If $N(A \cup B)$ is a clique, even though $\mfs(A, B) = \emptyset$, some bipartitions of $N(A \cup B)$ will not define a correct half-space separation of $A, B$.
This is illustrated in Figure \ref{fig:clique}.
\begin{figure}[ht!]
\centering
\includegraphics[scale=\FIGClique]{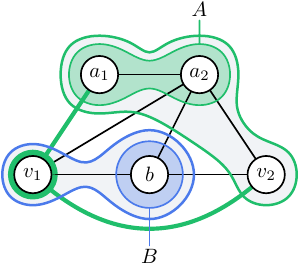}
\caption{A graph $G$ in which $N(A \cup B)$ is a clique, yet where the highlighted bipartition $A \cup v_2, B \cup v_1$ does not define half-spaces.
Since $F(A, v_1) \setminus F(A, v_2) \neq \emptyset$, $v_1 \in \cl(A \cup v_2$) holds by Proposition \ref{prop:convex} (an $a_1v_2$-path is highlighted in bold green).}
\label{fig:clique}
\end{figure}
In fact, unlike the case where $N(A \cup B)$ is not a clique, $F(A, v) = F(A, u)$ needs not hold for distinct vertices $u, v \in N(A \cup B)$.
Thus, Proposition \ref{prop:convex} may entail $u \in \cl(A \cup v)$ or vice-versa.
\end{remark}

On the other hand, when $N(A \cup B)$ is not a clique, the subsequent lemma implies that for any bipartition $X$, $Y$ of $N(A \cup B)$ into cliques, $A \cup X$ and $B \cup Y$ are convex.

\begin{lemma}[restate=LEMConvexClique, label=lem:convex-clique-1] 
Let $G$ be a connected graph and let $A, B$ be two linked, disjoint and saturated subsets of $V(G)$ such that $N(A \cup B)$ is not a clique.
Then for every clique $K \subseteq N(A \cup B)$, both $A \cup K$ and $B \cup K$ are convex.
\end{lemma}

\begin{proof}
To check that $A \cup K$ is convex, we verify that $J[u, v] \subseteq A \cup K$ for every $u, v \in A \cup K$.
If $u, v \in A$ or $u, v \in K$, then the result holds since $A$ is convex and $K$ is a clique.
Consider instead $u \in A, v \in K$.
Assume for contradiction $J[u, v] \nsubseteq K \cup A$.
There exists a chordless $uv$-path $u = v_1, \dots, v_k = v$ such that $v_i \notin K \cup A$ for some $1 < i < k$.
Consider the least such $i$.
By assumption $v_i \in N(A)$ and $v_{i-1} \in F(A, v_i)$.
Morever, since $A, B$ are saturated, $v_i \in N(A \cup B)$ must hold.
As $N(A \cup B)$ is not a clique, we obtain by Lemma \ref{lem:alternative} that $F(A, v_i) = F(A, v)$, meaning that $v_{i - 1}$ is adjacent to $v$.
This contradicts $v_i$ being on a chordless $uv$-path.
We deduce that $J[u, v] \subseteq A \cup K$ and $A \cup K$ is convex.
\end{proof}

We finally arrive at the following intermediate claim.

\begin{lemma}[restate=LEMIntermediateCobip, label=lem:intermediate-cobip] 
Let $G$ be a connected graph and let $A, B$ be two disjoint, linked and saturated subsets of $V(G)$.
If $\comp{A \cup B} = N(A \cup B)$, then $A$ and $B$ are separable if and only if $N(A \cup B)$ is cobipartite.
\end{lemma}

\begin{proof}
We start with the only if part.
Let $H = A \cup X$, $\comp{H} = B \cup Y$ be half-spaces separating $A$ and $B$.
By assumption, $X$ contains no forbidden pair of $\mfs(A, B)$.
Since $X \subseteq N(A \cup B)$, we deduce from Lemma \ref{lem:MFS} that $X$ is a clique.
In the same way, we deduce that $Y$ is a clique.
As $X, Y$ is a bipartition of $\comp{A \cup B} = N(A \cup B)$, we deduce that $N(A \cup B)$ is cobipartite.

We proceed to the if part.
If $N(A \cup B)$ is cobipartite, we have two cases: either $N(A \cup B)$ is a clique or it is not.
If $N(A \cup B)$ is a clique, then (resp.~$A \cup N(A \cup B)$ and $B$) are half-spaces separating $A$ and $B$.
If $N(A \cup B)$ is not a clique, the fact that $A \cup X$ and $B \cup Y$ are half-spaces for every bipartition $X, Y$ of $N(A \cup B)$ into cliques follows from Lemma \ref{lem:convex-clique-1}.
\end{proof}

Let us consider now that there are vertices outside of $N(A \cup B)$, i.e., $N(A \cup B) \subset \comp{A \cup B}$.
First, if $N(A \cup B)$ is a clique, $\mfs(A, B) = \emptyset$ still holds by Lemma \ref{lem:MFS}.
In this case, the same reasoning as before applies, and $A, B \cup \comp{A \cup B}$ is a half-space separation of $A$, $B$.
Suppose on the other hand that $N(A \cup B)$ is not a clique.
If it is not cobipartite, then any bipartition of $N(A \cup B)$ will contain a pair of non-adjacent vertices, and hence a forbidden pair, again due to Lemma \ref{lem:MFS}.
In other words, if $N(A \cup B)$ is not cobipartite, $A$ and $B$ are not separable.
However, there are also cases where $N(A \cup B)$ is cobipartite, yet $A$ and $B$ are not separable.
This is the case for the graphs of Figure \ref{fig:cobip-not-sufficient}, that we will use to illustrate the steps of the upcoming discussion.
\begin{figure}[ht!]
\centering
\includegraphics[page=1, scale=\FIGCobipNotSuf]{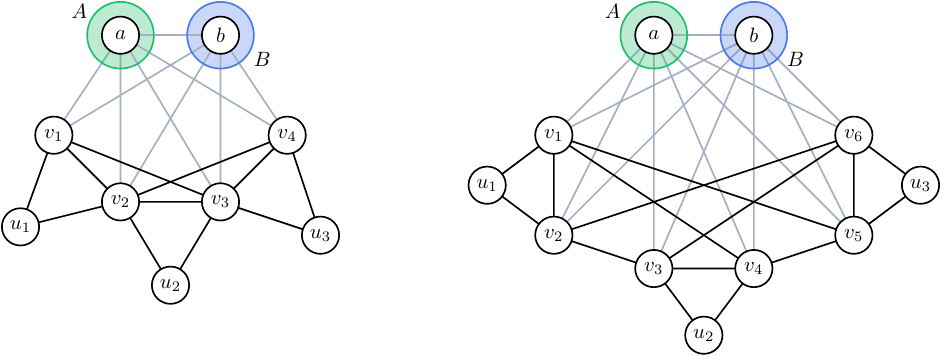}
\caption{Two examples where $A$ and $B$ are linked and saturated, yet not separable despite $N(A \cup B)$ being cobipartite. For readability, the edges incident to $a$ and $b$ are clearer. Remark that since $N(A \cup B)$ is not a clique, both $a$ and $b$ are complete to $N(A \cup B)$ in virtue of Lemma \ref{lem:alternative}.}
\label{fig:cobip-not-sufficient}
\end{figure}
This happens because when picking an element $v$ in a connected component $S$ of $G - N[A \cup B]$, $\cl(A \cup v)$ and $\cl(B \cup v)$ will share elements from $N(S)$, regardless of the structure of $N(A \cup B)$ (clique or not).

\begin{lemma}[restate=LEMConComp, label=lem:connected-components] 
Let $G$ be a connected graph and let $A, B$ be two linked, disjoint and saturated subsets of $V(G)$.
Let $S$ be a connected component of $G - N[A \cup B]$.
Then, for every $v \in S$, $N(S) \subseteq \cl(A \cup v) \cap \cl(B \cup v) \cap N(A \cup B)$.
\end{lemma}

\begin{proof}
Let $u \in N(S)$.
Because $S$ is a connected component of $G - N[A \cup B]$, $N(S) \subseteq N(A \cup B)$ holds by definition.
Thus, $u \in N(A \cup B)$.
Moreover, $S \cup u$ is connected in $G$.
Hence, there exists a chordless $uv$-path $u = v_1, \dots, v_k = v$.
Now, $F(A, B) \subseteq N(u) \setminus N(v)$ follows from Lemma \ref{lem:neighbors-equal} and the fact that $N(v) \cap A = \emptyset$, since $v \in S$.
Hence, for any $a \in F(A, B)$, $a, u, \dots, v_k = v$ is a chordless $av$-path.
We deduce $u \in \cl(A \cup v)$ for all $u \in N(S)$.
Since $A, B$ are linked and saturated, applying the same argument on $B$ yields $N(S) \subseteq \cl(A \cup v) \cap \cl(B \cup v) \cap N(A \cup B)$ as required.
\end{proof}

Using Lemma \ref{lem:connected-components}, we define an equivalence relation on $\comp{A \cup B}$ that will help us characterize the separability of $A$ and $B$.
Every half-space separation $H$, $\comp{H}$ of $A$ and $B$, if any, can be written as $H = A \cup X$ and $\comp{H} = B \cup Y$ where $X, Y$ is a bipartition of $\comp{A \cup B}$.
Since $H \cap \bar{H} = \emptyset$, we have $H \cap Y = \cl(A \cup X) \cap Y = \emptyset$ and similarly $\comp{H} \cap X = \cl(B \cup Y) \cap X = \emptyset$.
As a direct application of Lemma \ref{lem:connected-components}, we deduce:
\begin{enumerate}[(1)]
    \item For each connected component $S$ of $G - N[A \cup B]$, either $N[S] \subseteq X$ or $N[S] \subseteq Y$;
    \item If $S_1, \dots, S_k$ is a sequence of (not necessarily distinct) connected components of $G - N[A \cup B]$ such that $N(S_i) \cap N(S_{i+1}) \neq \emptyset$ for each $1 \leq i < k$, then $\bigcup_{i = 1}^k N[S_i]$ must be included in one of $X$ or $Y$.
    We call such a sequence an intersecting sequence of connected components.
\end{enumerate}
Given an intersecting sequence $S_1, \dots, S_k$ of connected components of $G - N[A \cup B]$, we say for brevity that $u, v$ belongs to the sequence $S_1, \dots, S_k$ if there exists $ 1 \leq i, j \leq k$ such that $u \in N[S_i]$ and $v \in N[S_j]$.  
Let us define the equivalence relation $\eq$ on $\comp{A \cup B}$ such that, for all $u, v \in \comp{A \cup B}$:
\begin{align*}
u \eq v \iff & u = v \text{ or } u, v \text{ belong to an intersecting sequence of } \\
& \text{connected components of } G - N[A \cup B]
\end{align*}

\begin{proposition}[restate=PROPEquivRel, label=prop:equivalence-relation] 
The relation $\eq$ is an equivalence relation.
\end{proposition}

\begin{proof}
The relation is symmetric and reflexive by definition.
To see that it is transitive, let $u, v, w$ be three vertices of $\comp{A \cup B}$ such that $u \eq v \eq w$.
If $u = v$ or $v = w$, $u \eq w$ readily holds.
Assume $u, v, w$ are distinct.
Let $S_1, \dots, S_k$ be an intersecting sequence with $u \in S_1$ and $v \in S_k$, and let $S'_1, \dots, S'_{\ell}$ be an intersecting sequence with $v \in S'_1$ and $w \in S'_{\ell}$.
If $v \notin N(A \cup B)$, then it belongs to a unique connected component of $G - N[A \cup B]$ so that $S_k = S'_1$.
If $v \in N(A \cup B)$ then $v \in N(S_k) \cap N(S'_1)$.
We deduce that $S_1, \dots, S_k, S'_1, \dots, S'_{\ell}$ is an intersecting sequence, and $u \eq w$ follows.
\end{proof}

\begin{remark}
The definition of $\eq$ encompasses the vertices that do not belong to the closed neighborhood of any connected component of $G - N[A \cup B]$, i.e., those vertices $v$ in $N(A \cup B)$ such that $N[v] \subseteq N[A \cup B]$.
By definition of $\eq$, they are equivalent to themselves only.
\end{remark}

\begin{figure}[ht!]
\centering
\includegraphics[scale=\FIGEqClasses, page=2]{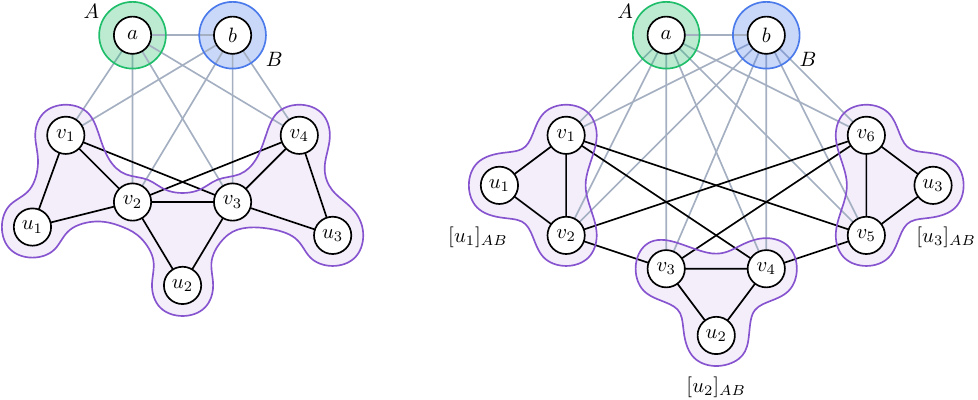}
\caption{The equivalence relation $\eq$ applied to the graphs of Figure \ref{fig:cobip-not-sufficient}.
The classes are circled (purple).
On the left, there is a unique equivalence class. Remark that, as a consequence, $v_1 v_4$ is a forbidden pair all the while $v_1 \eq v_4$.
On the right, there are three classes, $[u_1]_{AB}$, $[u_2]_{AB}$, and $[u_3]_{AB}$.}
\label{fig:eq-classes}
\end{figure}
For $u \in \comp{A \cup B}$, let $[u]_{AB}$ be the equivalence class of $u$: $[u]_{AB} = \{v \in \comp{A \cup B} \st u \eq v\}$.
In Figure \ref{fig:eq-classes}, we give the equivalence classes induced by $\eq$ on the graphs of Figure \ref{fig:cobip-not-sufficient}.
The next lemma is a direct yet important consequence of the above discussion and Lemma \ref{lem:connected-components}.

\begin{lemma}[restate=LEMEquivRel, label=lem:equivalence-relation] 
Let $G$ be a connected graph and let $A, B$ be two disjoint, linked and saturated subsets of $V(G)$.
For every bipartition $X, Y$ of $\comp{A \cup B}$, we have $\cl(A \cup X) \cap Y = \emptyset$ and $\cl(B \cup Y) \cap X = \emptyset$ only if for each $v \in \comp{A \cup B}$, either $[v]_{AB} \subseteq X$ or $[v]_{AB} \subseteq Y$.
\end{lemma}

We consider $\eq$ together with $\mfs(A, B)$.
Remind that $\mfs(A, B)$ consists in pairs of vertices only, thanks to Lemma \ref{lem:MFS}.
Hence, a forbidden pair $uv \in \mfs(A, B)$ falls into exactly one of the following cases regarding equivalence classes:
\begin{enumerate}[(1)]
    \item Either $u \eq v$ so that the equivalence class $[u]_{AB}$ prevents separation of $A$ and $B$ on its own (see Proposition \ref{prop:self-conflict} below).
     This case happens for instance in the graph on the left of Figure \ref{fig:eq-classes}: $v_1 \eq v_4$ yet $v_1v_4 \in \mfs(A, B)$.
    \item Or $u \not\eq v$, so that $[u]_{AB}$ and $[v]_{AB}$ cannot be taken together in any separation of $A$ and $B$.
    For example in the graph on the right of Figure \ref{fig:eq-classes} we have $v_1 \not\eq v_3$ and $v_1v_3 \in \mfs(A, B)$, which makes $[u_1]_{AB}$ and $[u_2]_{AB}$ incompatible for separating $A$ and $B$.
    In this example, all equivalence classes are incompatible, so that $A$ and $B$ not separable.
\end{enumerate}
As for the first case, we have the direct property:

\begin{proposition}[restate=PROPSelfConflict, label=prop:self-conflict] 
If $uv \in \mfs(A, B)$ and $u \eq v$, then $A$, $B$ are not separable.
\end{proposition}

\begin{proof}
Let $X, Y$ be a bipartition of $\comp{A \cup B}$ such that $\cl(A \cup X) \cap Y = \cl(B \cup Y) \cap X = \emptyset$, which must hold in order to find half-spaces separating $A$ and $B$.
By Lemma \ref{lem:equivalence-relation}, $[u]_{AB} \subseteq X$ or $[u]_{AB} \subseteq Y$.
since $uv \in \mfs(A, B)$ and $v \eq u$, we deduce either $\cl(X) \cap B \neq \emptyset$ or $\cl(Y) \cap B = \emptyset$.
Thus, for every bipartition $X, Y$ of $\comp{A \cup B}$, $\cl(A \cup X) \cap \cl(B \cup Y) \neq \emptyset$.
Hence, $A$ and $B$ are not separable.
\end{proof}

For the second case, we can build a graph $G_{AB}$ on the equivalence classes of $\eq$ that makes adjacent every two distinct equivalence classes sharing a forbidden pair.
More formally:
\begin{align*}
V(G_{AB}) = & \{[v]_{AB} \st v \in \comp{A \cup B}\} \\ 
E(G_{AB}) = & \{[u]_{AB} [v]_{AB} \st u \not\eq v \text{ and }uv \in \mfs(A, B) \}.
\end{align*}
For the graph on the right of Figure \ref{fig:eq-classes}, the corresponding graph $G_{AB}$ will be a clique.
Figure \ref{fig:eq-graph} illustrates the graph $G_{AB}$ on an other example.

\begin{remark}
In the case where $\comp{A \cup B} = N(A \cup B)$, the equivalence classes $[v]_{AB}$ are precisely the singletons $\{v\}$ for all $v \in N(A \cup B)$.
Identifying $[v]_{AB}$ with $v$, $G_{AB}$ turns out to be precisely the complement of $G[N(A \cup B)]$.
\end{remark}

\begin{figure}[ht!]
\centering
\includegraphics[page=1, scale=\FIGEqGraph]{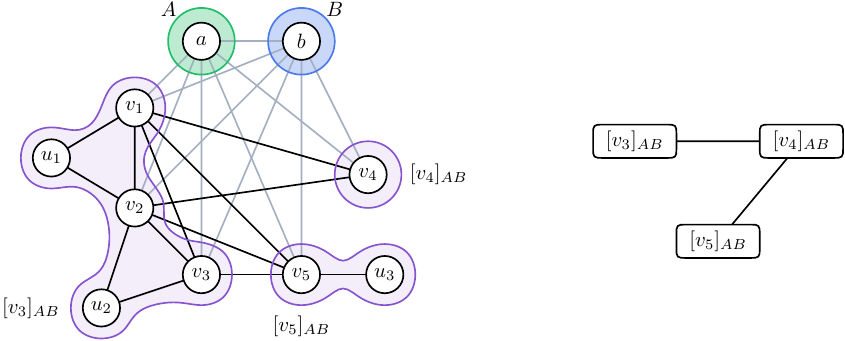}
\caption{On the left, a graph with linked $A$ and $B$ where the equivalence class are highlighted.
Again, the edges incident to $a$ and $b$ are clearer for readability.
On the right, the corresponding graph $G_{AB}$.}
\label{fig:eq-graph}
\end{figure}

Before characterizing the separability of $A$ and $B$ we give a lemma extending Lemma \ref{lem:convex-clique-1}.

\begin{lemma}[restate=LEMConvexCliqueTwo, label=lem:convex-clique-2] 
Let $G$ be a connected graph and let $A, B$ be two disjoint, linked and saturated subsets of $V(G)$ such that $N(A \cup B)$ is not a clique.
Then, for every collection $\cc{X}$ of equivalence classes of $\eq$, if $\bigcup \cc{X} \cap N(A \cup B)$ is a clique, then both $A \cup \bigcup \cc{X}$ and $B \cup \bigcup \cc{X}$ are convex.
\end{lemma}

\begin{proof}
Let $\cc{X}$ be a collection of equivalences classes such that $\bigcup \cc{X} \cap N(A \cup B)$ is a clique and let $C = A \cup \bigcup \cc{X}$.
We show that $C$ is convex.
We put $K = F(A, \bar A) \cup (N(A \cup B) \cap C)$.
Now, by assumption, $N(A \cup B) \cap C$ is a clique, $A$ and $B$ are linked and saturated and $N(A \cup B)$ is not a clique.
Therefore, $K$ is a clique by Corollary \ref{cor:central-clique}.
In view of Lemma \ref{lem:gonzales-separator}, we show that $K$ is a clique separator of $G$ and that $C \setminus K$ is a union of connected components of $G - K$.
First, since $F(A, \bar A) \subseteq K$ and $K$ is a clique, we have that $G - K$ disconnects $A \setminus K$ from $\comp{A \cup B} \setminus K$.
Hence $K$ is a clique separator of $G$ and moreover, $A \setminus K$ is indeed a union of connected components of $G - K$ since $F(A, \bar A) \subseteq K$.
Now we consider $C \setminus (K \cup A)$.
If $C \setminus (K \cup A) = \emptyset$, we deduce $C \subseteq A \cup K$ and the result holds by Lemma \ref{lem:convex-clique-1}.
Assume that $C \setminus (K \cup A) \neq \emptyset$ and let $S_1, \dots, S_k$ be the connected components of $G - N[A \cup B]$ such that $C \cap S_i \neq \emptyset$ for each $1 \leq i \leq k$.
We have $C \setminus (A \cup K) \subseteq \bigcup_{i = 1}^k S_i$.
We show that $\bigcup_{i = 1}^k S_i \subseteq C \setminus (A \cup K)$.
Let $v \in S_i$ for some $1 \leq i \leq k$.
By definition of $\eq$, $S \subseteq [v]_{AB}$ and since $\cc{X}$ is a collection of equivalence classes, we obtain $S \subseteq [v]_{AB} \subseteq C \setminus (K \cup A)$.
We deduce $C \setminus (A \cup K) \subseteq \bigcup_{i = 1}^k S_i$ and hence $C \setminus (A \cup K) = \bigcup_{i = 1}^k S_i$.
It remains to show that $S_i$ is a connected component of $G - K$.
Since $S_i$ is a connected component of $G - N[A \cup B]$, it is a connected component of $G - N(S_i)$.
Moreover, $N(S_i) \subseteq N(A \cup B)$ by construction.
Finally, again by definition of $C$ and $\eq$, $N(S_i) \subseteq C$.
Henceforth, $N(S_i) \subseteq N(A \cup B) \cap C$ from which we deduce that $S_i$ is a connected component of $G - K$.
Hence, $C = K \cup (A \setminus K) \cup (C \setminus (K \cup A))$ is the union of a clique separator $K$ of $G$ and connected components of $G - K$.
Applying Lemma \ref{lem:gonzales-separator}, we deduce that $C$ is convex, which concludes the proof. 
\end{proof}

We can characterize the separability of $A, B$ by generalizing Lemma \ref{lem:intermediate-cobip}.

\begin{theorem}[restate=THMLinkedSep, label=thm:linked-separability] 
Let $G$ be a connected graph and let $A, B$ be two disjoint, linked and saturated subsets of $V$.
Then $A$ and $B$ are separable if and only if the next conditions hold:
\begin{enumerate}[(1)]
    \item for every $v \in \comp{A \cup B}$, $[v]_{AB}$ contains no forbidden pairs; 
    \item $G_{AB}$ is bipartite.
\end{enumerate}
\end{theorem}

\begin{proof}
We start with the only if part.
Assume $A$ and $B$ are separable and let $H, \bar{H}$ be a half-space separation of $A$ and $B$ with $A \subseteq H$ and $B \subseteq \bar H$.
Put $X = H \setminus A$ and $Y = \comp{H} \setminus B$.
By assumption, $H \cap Y = \cl(A \cup X) \cap Y = \emptyset$. 
Hence, by Lemma \ref{lem:equivalence-relation}, for each $v \in \comp{A \cup B}$, either $[v]_{AB} \subseteq X$ or $[v]_{AB} \subseteq Y$.
Let $\cc{X} = \{[v]_{AB} \in V(G_{AB}) \st [v]_{AB} \subseteq X\}$ and $\cc{Y} = \{[v]_{AB} \in V(G_{AB}) \st [v]_{AB} \subseteq Y\}$.
Since $H, \comp{H}$ are half-spaces separating $A$ and $B$, and $X \subseteq H$, $Y \subseteq \comp{H}$, we deduce that neither $X$ nor $Y$ contain a forbidden pair of $\mfs(A, B)$.
We derive:
\begin{enumerate}[(1)]
    \item for each $[v]_{AB}$, $[v]_{AB}$ contains no forbidden pair, i.e., item (1) holds;
    \item for each pair of distinct classes $[u]_{AB}, [v]_{AB}$ in $X$ (resp.~$Y$), $[u]_{AB}$ and $[v]_{AB}$ are not adjacent in $G_{AB}$, i.e., that $\cc{X}$ (resp.~$\cc{Y}$) is an independent set of $G_{AB}$.
    Since $\cc{X}$, $\cc{Y}$ is a partition of $G_{AB}$ into two independent sets, we conclude that $G_{AB}$ is bipartite, and that item (2) of the theorem holds.
\end{enumerate}

We move to the if part.
Assume both items (1) and item (2) are satisfied.
In particular, if $N(A \cup B)$ is a clique, $\mfs(A, B) = \emptyset$ by Lemma \ref{lem:MFS}.
Hence, $A \cup N(A \cup B)$ and $B$ (resp.~$B \cup N(A \cup B)$ and $A$) are half-spaces separating $A$ and $B$.
Assume $N(A \cup B)$ is not a clique and let $\cc{X}, \cc{Y}$ be any bipartition of $V(G_{AB})$ into two independent sets.
We show that $\bigcup \cc{X}$ contains no forbidden pair.
Assume for contradiction there exists a forbidden pair $uv \in \bigcup \cc{X}$. 
We have two cases:
\begin{enumerate}[(1)]
    \item $u \eq v$, but this would contradict item (1) of the statement;
    \item $u \not\eq v$, but this would contradict $\cc{X}_{A}$ being an independent set of $G_{AB}$ by definition of $G_{AB}$.
\end{enumerate}
By Lemma \ref{lem:MFS} we deduce that $\bigcup \cc{X}$ contains no forbidden pair, and hence that $\bigcup \cc{X} \cap N(A \cup B)$ is a clique.
Applying Lemma \ref{lem:convex-clique-2}, $A \cup \bigcup \cc{X}$ is convex.
The same reasoning on $B \cup \cc{Y}$ yields that $A \cup \bigcup \cc{X}$ and $B \cup \cc{Y}$ are half-spaces separating $A$ and $B$.
This concludes the proof.
\end{proof}

Figure \ref{fig:eq-graph-separability} illustrate the conditions of Theorem \ref{thm:linked-separability} on the example of Figure \ref{fig:eq-graph}.
\begin{figure}[ht!]
\centering
\includegraphics[scale=\FIGEqGraphSep, page=2]{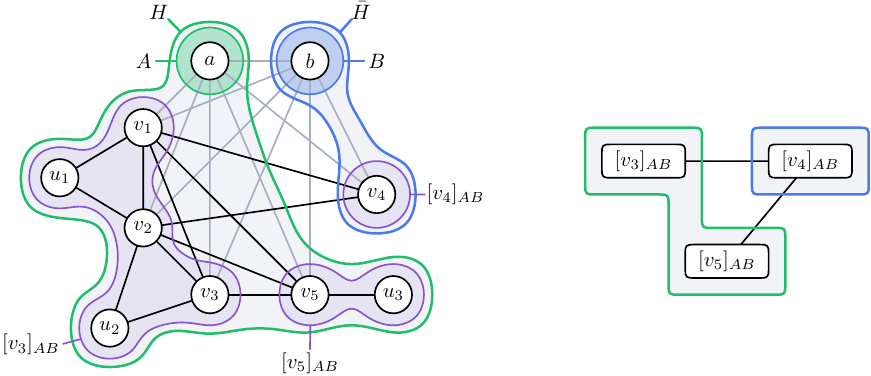}
\caption{Illustration of Theorem \ref{thm:linked-separability} on the graph of Figure \ref{fig:eq-graph}.
A half-space separation of $A$ and $B$ is drawn.
Observe that it corresponds to a bipartition of $G_{AB}$ into independent sets.}
\label{fig:eq-graph-separability}
\end{figure}
We finally argue that the conditions of Theorem \ref{thm:linked-separability} can be checked in polynomial time.
Since $\mfs(A, B)$ consists in pairs only, it can be computed in polynomial time.
Then, we identify the connected components of $G - N[A \cup B]$ in polynomial time by traversing $G - N[A \cup B]$.
We then identify the equivalence relation $\eq$ and build $G_{AB}$ accordingly.
Testing that no equivalent class contains a forbidden pair can be done in polynomial as well as checking that $G_{AB}$ is bipartite also takes polynomial time.
We deduce:

\begin{theorem}[restate=THMLinkedSepPoly, label=thm:linked-separability-poly] 
Let $G$ be a connected graph and let $A, B$ be two disjoint, linked and saturated subsets of $V$.
Whether $A, B$ can be separated by half-spaces can be checked in polynomial time in the size of $G$.
\end{theorem}

\section{Conclusion}
\label{sec:conclusion}

We proved that half-space separability can be tested in polynomial time for monophonic convexity.
Using Lemma \ref{lem:convex-clique-2}, the algorithm we propose can be adapted to generate a pair of half-spaces separating two sets of vertices, if any.
Moreover, we deduce as a corollary that the $2$-partition problem can be solved in polynomial time for monophonic convexity, thus answering an open problem in \cite{gonzalez2020covering}.

To decide separability, we used the underlying graph in conjunction with the Carathéodory number, being constant for monophonic convexity \cite{duchet1988convex}.
A natural question is then to investigate to what extent the Carathéodory number can be used to decide separability.
However, relying on the problem of $2$-coloring $3$-uniform hypergraphs \cite{lovasz1973coverings}, we can show that already with Carathéodory number 3, half-space separation in general convexity spaces is out of reach.

\begin{theorem}[restate=THMHardnessCara, label=thm:hardness-cara]
Half-space separation is $\NP$-complete for convexity spaces $(V, \cs)$ given by $V$ and a hull operator $\cl$ that computes $\cl(X)$ in polynomial time in the size of $V$ for all $X \subseteq V$, even if $(V, \cs)$ has Carathéodory number 3.
\end{theorem}

\begin{proof}
We first show the problem belongs to $\NP$.
Let $(V, \cs)$ be given by $V$ and the hull operator $\cl$, and let $A, B \subseteq V$.
A certificate is a pair of half-spaces $H, \comp{H}$ separating $A$ and $B$.
This can be tested efficiently by checking $A \subseteq \cl(H) = H$ and $B \subseteq \cl(\comp{H}) = \comp{H}$.

To show hardness, we use a reduction from $3$-uniform hypergraph $2$-colouring.
A hypergraph $\H$ is a pair $(V(\H), \E(\H))$ where $\E(\H)$ is a non-empty collection of subsets of $V(\H)$ called (hyper)edges.
The hypergraph $\H$ is $k$-uniform if all the edges have size $k$.
An independent set of $\H$ is a set $I \subseteq V(\H)$ such that $E \nsubseteq I$ for each $E \in \E(\H)$.
\begin{decproblem}
    \problemtitle{3-uniform hypergraph 2-coloring}
    \probleminput{A $3$-uniform hypergraph $\H$.}
    \problemquestion{is there a non-trivial bipartition of $V(\H)$ into independent sets ?} 
\end{decproblem}
The problem has been shown $\NP$-complete by Lovàsz \cite{lovasz1973coverings}.

We build a reduction towards half-space separation.
Let $\H = (V(\H), \E(\H))$ be a $3$-uniform hypergraph.
Let $V' = V(\H) \cup \{a, b\}$, $A = \{a\}$, $B = \{b\}$.
We consider the convexity space $(V', \cs)$ with hull operator $h$ defined for $X \subseteq V'$ as follows:
\[ 
h(X) = \begin{cases}
    X  \cup \{a, b\}& \text{ if } X \text{ includes an edge of } \E(\H) \\
    X & \text{otherwise.}
\end{cases}
\]
Since $\H$ is $3$-uniform, $\E(\H)$ has size at most $\card{V(\H)}^3$ and $\cl$ can be computed in polynomial time in the size of $V$ by scanning hyperedges.
We show that $(V', \cs)$ has Carathéodory number $3$.
Let $X \subseteq V(\H)$ and $v \in V(\H)$ such that $v \in \cl(X)$.
By definition of $\cl$, if $v \in V(\H)$, then $v \in X$ must hold.
On the other hand, if $v = a$ or $v = b$ and $v \notin X$, then $v \in \cl(X)$ if and only if $X$ contains en edge of $\E(\H)$.
Since edges of $\H$ have size $3$, we deduce that for each $X$, $v$ such that $v \in \cl(X)$, the least $d$ such that $v \in \cl(Y)$ with $Y \subseteq X$ and $\card{Y} \leq d$ is $3$.
This proves that the Carathéodory number of $(V', \cs)$ is $3$ as expected.

Now, we readily have by definition of $\cl$ that $A$ and $B$ are separable if and only if there exists a bipartition $I$, $J$ of $V' \setminus \{a, b\}$ such that $I, J$ are independent sets of $\H$.
Henceforth, an algorithm that runs in polynomial time in the size of $V$ and that uses only a polynomial number of calls to $\cl$ will solve the input instance of $3$-uniform hypergraph $2$-coloring in polynomial time.
This concludes the proof. 
\end{proof}

Theorem \ref{thm:hardness-cara} together with Theorem \ref{thm:separability-polynomial} motivate the next intriguing open problem.

\begin{problem}
Find a natural (graph) convexity with Carathéodory number $2$ (e.g., triangle-path convexity \cite{changat1999triangle}) where half-space separation is hard, or show that for all such convexities, half-space separation is tractable.
\end{problem}

\userpar{Funding} The first two authors have been funded by the CMEP 
Tassili project: "Argumentation et jeux: Structures et Algorithmes", Codes: 46085QH-21MDU320 PHC, 2021-2023.
This research is also supported by the French government IDEXISITE initiative 
16-IDEX-0001 (CAP 20-25).

\userpar{Acknowledgements} We thank anonymous reviewers for their comments. We are also grateful to Victor Chepoi for suggesting the term ``shadow'' instead of ``extension'' as well as pointing us to further references, especially the recent works on monophonic convexity \cite{bressan2024efficient, chepoi2024separation}.

\bibliographystyle{alpha}
\bibliography{biblio}	

\end{document}